\documentclass{amsart}
\usepackage{amsmath}
\usepackage{amscd}
\usepackage{amssymb}
\usepackage{amsthm}
\usepackage{array}
\usepackage{fullpage}
\usepackage[numbers]{natbib}  
\theoremstyle{plain}
\newtheorem{prop}{Proposition}
\newtheorem{thm}[prop]{Theorem}
\newtheorem{lem}[prop]{Lemma}
\newtheorem{cor}[prop]{Corollary}

\numberwithin{prop}{section}

\usepackage{hyperref}
\hypersetup{colorlinks,linkcolor={red},citecolor={blue},urlcolor={blue}}

\newcolumntype{m}{>{$}c<{$}}

\theoremstyle{definition}
\newtheorem{defn}[prop]{Definition}
\newtheorem{rem}[prop]{Remark}

\newenvironment{psmallmatrix}
  {\left(\begin{smallmatrix}}
  {\end{smallmatrix}\right)}

\begin{document}

\title{Siegel modular forms of degree two and level five}

\author{Haowu Wang}

\address{Center for Geometry and Physics, Institute for Basic Science (IBS), Pohang 37673, Korea}

\email{haowu.wangmath@gmail.com}

\author{Brandon Williams}

\address{Lehrstuhl A für Mathematik, RWTH Aachen, 52056 Aachen, Germany}

\email{brandon.williams@matha.rwth-aachen.de}

\subjclass[2010]{11F46,11F27}

\date{\today}

\keywords{Siegel modular forms, Borcherds lifts, Rings of modular forms}

\begin{abstract}
We construct a ring of meromorphic Siegel modular forms of degree $2$ and level $5$, with singularities supported on an arrangement of Humbert surfaces, which is generated by four singular theta lifts of weights $1$, $1$, $2$, $2$ and their Jacobian. We use this to prove that the ring of holomorphic Siegel modular forms of degree $2$ and level $\Gamma_0(5)$ is minimally generated by eighteen modular forms of weights $2, 4, 4, 4, 4, 4, 6, 6, 6, 6, 10, 11, 11, 11, 13, 13, 13, 15$.
\end{abstract}

\maketitle

\section{Introduction}

It is an interesting problem to determine the structure of rings of Siegel modular forms with respect to congruence subgroups. A famous theorem of Igusa \cite{I62} shows that every Siegel modular form of degree two and even weight for the full modular group $\mathrm{Sp}_4(\mathbb{Z})$ can be written uniquely as a polynomial in forms $\phi_4, \phi_6, \phi_{10}, \phi_{12}$ of weights $4, 6, 10, 12$, and that odd weight Siegel modular forms are precisely the products of even weight Siegel modular forms with a distinguished cusp form $\psi_{35}$ of weight $35$. It was proved by Aoki and Ibukiyama \cite{AI} that the rings of modular forms for the congruence subgroups $$\Gamma_{0, 1}^{(2)}(N) = \Big\{ \begin{psmallmatrix} a & b \\ c & d \end{psmallmatrix} \in \mathrm{Sp}_4(\mathbb{Z}): \; c \equiv 0 \, (N), \; \mathrm{det}(a) \equiv \mathrm{det}(d) \equiv 1 \, (N)\Big\}, \; N = 2, 3, 4$$ have an analogous structure: they are generated by four algebraic independent modular forms together with their Jacobian (or first Rankin--Cohen--Ibukiyama bracket). The rings $M_*(\Gamma_0^{(2)}(N))$ where $$\Gamma_0^{(2)}(N) = \{\begin{psmallmatrix} a & b \\ c & d \end{psmallmatrix}\in \mathrm{Sp}_4(\mathbb{Z}): \, c \equiv 0 \, (N)\}$$ therefore have a simple structure as well.

The goal in this paper is to extend the methods of Aoki and Ibukiyama to level $N=5$. This is not quite straightforward, as the natural underlying ring is no longer  $M_*(\Gamma_{0,1}^{(2)}(5))$ but rather a ring $M_*^!(\Gamma_{0,1}^{(2)}(5))$ of meromorphic Siegel modular forms with singularities on Humbert surfaces. We will define a hyperplane arrangement $\mathcal{H}$ as the $\Gamma_0^{(2)}(5)$-orbit of the Humbert surface $$\{Z = \begin{psmallmatrix} \tau & z \\ z & w \end{psmallmatrix} \in \mathbb{H}_2: \; \mathrm{det}(Z) = 1 - 5z\},$$ 
which, if one views points in $\mathbb{H}_2$ as parameterizing abelian surfaces, is a locus of principally polarized abelian surfaces with real multiplication that respects a $\Gamma_0^{(2)}(5)$ level structure. We then investigate the ring $M_*^!(\Gamma_{0, 1}^{(2)}(5))$ of meromorphic Siegel modular forms on $\Gamma_{0, 1}^{(2)}(5)$ with singularities supported on $\mathcal{H}$.    Using a generalization of the modular Jacobian approach of \cite{Wan21a}, we prove in Theorem \ref{th:singular1} that $M_*^!(\Gamma_{0, 1}^{(2)}(5))$ is generated by four algebraically independent singular additive lifts $f_1, f_2, g_1, g_2$ of weights $1$, $1$, $2$, $2$ and by their Jacobian; in particular, the associated threefold $X_{0, 1}^{(2)}(5)$ is rational.  The local isomorphism from $\mathrm{Sp}_4$ to $\mathrm{SO}(3,2)$ and Borcherds' theory of orthogonal modular forms with singularities are essential. $\mathrm{Proj}(M_*^!(\Gamma_{0, 1}^{(2)}(5)))$ is the Looijenga compactification \cite{Loo03b} of the complement $(\mathbb{H}_2 \setminus \mathcal{H}) / \Gamma_{0, 1}^{(2)}(5)$, which plays a similar role to the Satake--Baily--Borel compactification of $Y_{0,1}^{(2)}(5)$. 


It follows from the above that every Siegel modular form of level $\Gamma_0^{(2)}(5)$ can be expressed uniquely in terms of the basic forms $f_1, f_2, g_1, g_2$. It is not clear to the authors how to compute the ring $M_*(\Gamma_0^{(2)}(5))$ of (holomorphic) Siegel modular forms from this information alone; however, allowing a formula of Hashimoto \cite{H} for the dimensions of cusp forms (itself an application of the Selberg trace formula), the ring structure becomes a straightforward Gr\"obner basis computation. We will prove that $M_*(\Gamma_0^{(2)}(5))$ is minimally generated by eighteen modular forms of weights $2, 4, 4, 4, 4, 4, 6, 6, 6, 6, 10, 11, 11, 11, 13, 13, 13, 15$ in Theorem \ref{thm:lvl5hol}.

This paper is organized as follows. In \S \ref{sec:pre} we review the realization of Siegel modular groups as orthogonal groups and the theory of Borcherds lifts. In \S \ref{sec:meromorphic} we determine two rings of meromorphic Siegel modular forms. In \S \ref{sec:holomorphic} we use this to determine the ring of holomorphic Siegel modular forms for $\Gamma_{0}^{(2)}(5)$.  

\section{Theta lifts to Siegel modular forms of degree two}\label{sec:pre}

\subsection{$\Gamma_0^{(2)}(N)$ as an orthogonal group}
Recall that the Pfaffian of an antisymmetric $(4 \times 4)$-matrix $M$ is $$\mathrm{pf}(M) = \mathrm{pf}\begin{psmallmatrix} 0 & a & b & c \\ -a & 0 & d & e \\ -b & -d & 0 & f \\ -c & -e & -f & 0 \end{psmallmatrix} = af - be + cd.$$ We view $\mathrm{pf}$ as a quadratic form and define the associated bilinear form $$\langle x, y \rangle = \mathrm{pf}(x + y) - \mathrm{pf}(x) - \mathrm{pf}(y).$$ The Pfaffian is invariant under conjugation $M \mapsto A^T M A$ by $A \in \mathrm{SL}_4(\mathbb{R})$ and this action identifies $\mathrm{SL}_4(\mathbb{R})$ with the Spin group $\mathrm{Spin}(\mathrm{pf})$. The symplectic group $\mathrm{Sp}_4(\mathbb{R})$, by definition, preserves $$\mathcal{J} = \begin{psmallmatrix} 0 & 0 & -1 & 0 \\ 0 & 0 & 0 & -1 \\ 1 & 0 & 0 & 0 \\ 0 & 1 & 0 & 0 \end{psmallmatrix}$$ under conjugation, so it also preserves the orthogonal complement $\mathcal{J}^{\perp}$; and indeed it is exactly the Spin group of $\mathrm{pf}$ restricted to $\mathcal{J}^{\perp}$. If the entries of $M$ are labelled as above then $M \in \mathcal{J}^{\perp}$ if and only if $b + e = 0$.

For any $N \in \mathbb{N}$, the group $\Gamma_0^{(2)}(N) = \{\begin{psmallmatrix} a & b \\ c & d \end{psmallmatrix} \in \mathrm{Sp}_4(\mathbb{Z}): \; c \equiv 0 \, (N)\}$ stabilizes the lattice $$L = \Big\{ M = \begin{psmallmatrix} 0 & a & b & c \\ -a & 0 & d & -b \\ -b & -d & 0 & f \\ -c & b & -f & 0 \end{psmallmatrix}: \; a, b, c, d, f \in \mathbb{Z}, \; a \equiv 0 \, (N)\Big\}$$ which is of type $U \oplus U(N) \oplus A_1$.
By \cite[\S 2]{HK} the special discriminant kernel $\widetilde{\mathrm{SO}}(L)$ of $L$ is exactly the projective modular group $\Gamma_{0, 1}^{(2)}(N) / \{\pm I\}$ under this identification, where
\begin{align*}
\widetilde{\mathrm{SO}}(L)&=\{ g \in \mathrm{SO}(L):\; g(v)-v\in L \; \text{for all}\; v \in L' \},\\
\Gamma_{0, 1}^{(2)}(N) &= \{\begin{psmallmatrix} a & b \\ c & d \end{psmallmatrix} \in \Gamma_0^{(2)}(N): \; \mathrm{det}(a) \equiv 1 \, (N)\}.
\end{align*}
It follows that $$\Gamma_0^{(2)}(N) / \{\pm I\} = \langle \widetilde{\mathrm{SO}}(L), \; \varepsilon_u: \; u \in (\mathbb{Z} / N\mathbb{Z})^{\times} \rangle$$ where $\varepsilon_u$ is the matrix $$\varepsilon_u = \begin{pmatrix} u & 0 & b & 0 \\ 0 & 1 & 0 & 0 \\ N & 0 & u^* & 0 \\ 0 & 0 & 0 & 1 \end{pmatrix} \in \Gamma_0^{(2)}(N)$$ for any integer solutions $u^*, b$ to $uu^* - Nb = 1$ (the choice does not matter). The map induced by $\varepsilon_u$ on $L'/L \cong A_1'/A_1 \oplus U(N)'/U(N)$ acts trivially on $A_1'/A_1$ and acts on $U(N)'/U(N) \cong \mathbb{Z}/N\mathbb{Z} \oplus \mathbb{Z}/N\mathbb{Z}$ as the map $$\varepsilon_u : \mathbb{Z}/N\mathbb{Z} \oplus \mathbb{Z}/N\mathbb{Z} \rightarrow \mathbb{Z}/N\mathbb{Z} \oplus \mathbb{Z}/N\mathbb{Z}, \; (x, y) \mapsto (ux, u^{-1}y).$$

The symplectic group $\mathrm{Sp}_4(\mathbb{R})$ acts on the Siegel upper half-space $\mathbb{H}_2$ by M\"obius transformations: $$M \cdot Z = \begin{pmatrix} a & b \\ c & d \end{pmatrix} \cdot Z = (aZ + b)(cZ + d)^{-1}.$$ Let $j(M; Z) = \mathrm{det}(c Z + d)$ be the usual automorphy factor. We embed the Siegel upper half-space into $L \otimes \mathbb{C}$ as follows: $$Z = \begin{pmatrix} \tau & z \\ z & w \end{pmatrix} \mapsto \mathcal{Z} := \phi(Z) := \begin{pmatrix} 0 & 1 & z & w \\ -1 & 0 & -\tau & -z \\ -z & \tau & 0 & \tau w - z^2 \\ -w & z & z^2 - \tau w & 0 \end{pmatrix}.$$ Then one has the relation $$M^T \mathcal{Z} M = j(M; Z) \phi(M \cdot Z), \; M \in \mathrm{Sp}_4(\mathbb{R})$$ as one can check on any system of generators.

For any $\lambda \in L'$ of positive norm $D = Q(\lambda)$, the space $$\{Z \in \mathbb{H}_2: \; \mathcal{Z} \in \lambda^{\perp}\}$$ is known as a \textbf{Humbert surface} $H(D, \lambda)$ of discriminant $D$. If $\lambda$ is written in the form $\begin{psmallmatrix} 0 & a & b & c \\ -a & 0 & d & -b \\ -b & -d & 0 & f \\ -c & b & -f & 0 \end{psmallmatrix}$ then $$H(D, \lambda) = \{Z = \begin{psmallmatrix} \tau & z \\ z & w \end{psmallmatrix} \in \mathbb{H}_2: \; a \mathrm{det}(Z) -c \tau + 2bz + dw + f = 0 \}.$$ If $\gamma$ instead is a coset of $L'/L$, then we define $$H(D, \gamma) = \sum_{\substack{\lambda \in \gamma \\ \lambda \; \text{primitive in}\; L' \\ Q(\lambda) = D}} H(D, \lambda).$$ These unions are locally finite and therefore descend to well-defined divisors on $\widetilde{\mathrm{O}}(L) \backslash \mathbb{H}_2.$ We will use the notation $H(D, \pm \gamma)$ because $\lambda^{\perp} = (-\lambda)^{\perp}$ implies $H(D, \gamma) = H(D, -\gamma)$. Note that many references omit the condition that $\lambda$ be primitive in $L'$, so $H(D, \pm \gamma)$ satisfy inclusions; our divisors $H(D, \pm \gamma)$ do not.

\subsection{Theta lifts}
Let $L$ be the lattice in the space of $(4 \times 4)$ antisymmetric matrices from the previous subsection. The \textbf{weight $k$ theta kernel} is $$\Theta_k(\tau; Z) = \frac{\pi^k}{\mathrm{det}(V)^k \Gamma(k)} \sum_{\lambda \in L'} \langle \lambda, \mathcal{Z} \rangle^k e^{-\frac{\pi y}{\mathrm{det}(V)} | \langle \lambda, \mathcal{Z} \rangle|^2} e^{2\pi i \overline{\tau} \mathrm{pf}(\lambda)} \mathfrak{e}_{\lambda},$$ where $\tau = x+iy \in \mathbb{H}$; and $Z = U+iV \in \mathbb{H}_2$; and $\mathcal{Z}$ is the image of $Z$ in $L \otimes \mathbb{C}$. By applying a theorem of Vign\'eras on indefinite theta series \cite{V} one sees that $\Theta_k$ transforms like a modular form of weight $\kappa:= k -1/2$ with respect to the Weil representation $\rho_L$. On the other hand, for any $M = \begin{psmallmatrix} a & b \\ c & d \end{psmallmatrix} \in \Gamma_0^{(2)}(N)$,
\begin{align*} \Theta_k(\tau; M \cdot Z) &= \frac{\pi^k}{\mathrm{det} \, \mathrm{im}(M \cdot Z)^k \Gamma(k)} \sum_{\lambda \in L'} \mathrm{det}(cZ + d)^{-k} \langle \lambda, M^T \mathcal{Z} M \rangle e^{-\frac{\pi y}{\mathrm{det}(V)} | \langle \lambda, M^T \mathcal{Z} M \rangle|^2} e^{2\pi i \overline{\tau} \mathrm{pf}(\lambda)} \mathfrak{e}_{\lambda} \\ &= \frac{\pi^k}{\mathrm{det}(V)^k \Gamma(k)} \overline{\mathrm{det}(cZ + d)^k} \sum_{\lambda \in L'} \langle M^{-T} \lambda M^{-1}, \mathcal{Z} \rangle^k e^{-\frac{\pi y}{\mathrm{det}(V)} | \langle M^{-T} \lambda M^{-1},  \mathcal{Z} \rangle|^2} e^{2\pi i \overline{\tau} \mathrm{pf}(M^{-T} \lambda M^{-1})} \mathfrak{e}_{\lambda} \\ &= \overline{\mathrm{det}(cZ + d)^k} \sigma(M) \Theta_k(\tau; Z),   \end{align*} where $\sigma$ is the map $$\sigma : \Gamma_0^{(2)}(N) \longrightarrow \mathrm{Aut} \, \mathbb{C}[L'/L], \; \; \sigma(M) \mathfrak{e}_{\lambda} := \mathfrak{e}_{M^T \lambda M}.$$

Following Borcherds \cite{B} one defines the theta lift of a vector-valued modular form $F$ with a pole at $\infty$ as the regularized integral of $F$ against the kernel $\Theta_k$:

\begin{defn} (i) The \emph{Weil representation} $\rho_L$ associated to an even lattice $(L, Q)$ is the representation $\rho : \mathrm{Mp}_2(\mathbb{Z}) \rightarrow \mathrm{GL}\, \mathbb{C}[L'/L]$ defined by $$\rho\left( \begin{psmallmatrix} 1 & 1 \\ 0 & 1 \end{psmallmatrix}, 1 \right) \mathfrak{e}_{\gamma} = e^{-2\pi i Q(\gamma)}\mathfrak{e}_{\gamma};$$ $$\rho\left( \begin{psmallmatrix} 0 & -1 \\ 1 & 0 \end{psmallmatrix}, \sqrt{\tau} \right) \mathfrak{e}_{\gamma} = e^{\pi i \mathrm{sig}(L)/4} |L'/L|^{-1/2} \sum_{\beta \in L'/L} e^{2\pi i \langle \gamma, \beta \rangle} \mathfrak{e}_{\beta}.$$ Here $\mathrm{Mp}_2(\mathbb{Z})$ is the metaplectic group of pairs $(M, \phi)$ where $M = \begin{psmallmatrix} a & b \\ c & d \end{psmallmatrix} \in \mathrm{SL}_2(\mathbb{Z})$ and $\phi$ is a square root of $c\tau + d$, and $\mathfrak{e}_{\gamma}$, $\gamma \in L'/L$ is the standard basis of the group ring $\mathbb{C}[L'/L]$.\\
(ii) A \emph{nearly-holomorphic vector-valued modular form} for $\rho_L$ is a holomorphic function $F : \mathbb{H} \rightarrow \mathbb{C}[L'/L]$ which satisfies $$F((M, \phi) \cdot \tau) = \phi(\tau)^{2k} \rho_L(M) F(\tau), \; \; (M, \phi) \in \mathrm{Mp}_2(\mathbb{Z})$$ and which is meromorphic at the cusp $\infty$, i.e. its Fourier series has only finitely many negative exponents. \\
(iii) Let $k \ge 1$ and let $F \in M_{\kappa}^!(\rho_L)$ be a nearly-holomorphic modular form. The \textbf{(singular) theta lift} of $F$ is $$\Phi_F(Z) = \int^{\mathrm{reg}}_{\mathrm{SL}_2(\mathbb{Z}) \backslash \mathbb{H}} \langle F(\tau), \Theta_k(\tau; Z) \rangle y^{\kappa} \frac{\mathrm{d}x \, \mathrm{d}y}{y^2}.$$
\end{defn}

Here the regularization means one takes the limit as $w \rightarrow \infty$ of the integral over $\mathcal{F}_w = \{\tau = x+iy \in \mathbb{H}: \; x^2 + y^2 \ge 1, \; |x| \le 1/2, \; y \le w\}$; in effect, it means one integrates first with respect to $x$, which mollifies the contribution of the principal part of $F$ to the integral; and then secondly with respect to $y$. The behavior of the theta lift under M\"obius transformations is \begin{align*} \Phi_F(M \cdot Z) &= \int^{\mathrm{reg}} \langle F(\tau), \Theta_k(\tau; M \cdot Z) \rangle \; y^{\kappa - 2} \, \mathrm{d}x \, \mathrm{d}y \\ &= \mathrm{det}(cZ + d)^k \int^{\mathrm{reg}} \sum_{\gamma \in L'/L} F_{\gamma}(\tau) \overline{\Theta_{k; M^{-T} \gamma M^{-1}}(\tau; Z)} \, y^{\kappa - 2} \, \mathrm{d}x \, \mathrm{d}y \\ &= \mathrm{det}(cZ+d)^k \Phi_{\sigma(M)^{-1} F}(Z). \end{align*} 
Therefore, the singular theta lift $\Phi_F$ defines a meromorphic Siegel modular form of weight $k$ on the subgroup of $\Gamma_0^{(2)}(N)$ that fixes $F$, with singularities of multiplicity $k$ along Humbert surfaces associated to the principal part of the input $F$. This is the so-called Borcherds additive lift. We refer to \cite[Theorem 14.3]{B} for more details. Since the Borcherds additive lift is a generalization of the Gritsenko lift \cite{G88}, we also call it the singular Gritsenko lift.  When the input $F$ has weight $\kappa=-\frac{1}{2}$ (i.e. $k=0$), the modified exponential of $\Phi_F$ defines a remarkable modular form which has an infinite product expansion (cf. \cite[Theorem 13.3]{B}), called a \emph{Borcherds product}, or more specifically the Borcherds lift of $F$. In this paper we will need both types of singular theta lifts.

\begin{rem} It is often useful to consider the \emph{pullback} or restriction of a Siegel modular form to a Humbert surface, the result being traditionally interpreted as a Hilbert modular form attached to a real-quadratic field. From the point of view of orthogonal modular forms this is very simple: to restrict a form $\Phi(Z)$ to the sublattice $v^{\perp}$ (with $v \in L'$) one simply restricts to $Z$ satisfying $\langle Z, v \rangle = 0$.

The pullback of a theta lift $\Phi_F$ as above is again a theta lift, $\Phi_{\vartheta F}$, where $\vartheta F \in M_{\kappa + 1/2}(\rho_{v^{\perp}})$ is the \emph{theta contraction}, obtained roughly by tensoring $F$ with a unary theta series and averaging out. The important point is that one can check rigorously whether a theta lift $\Phi_F$ vanishes identically on a Heegner divisor, with the computations taking place only on the level of vector-valued modular forms.
\end{rem}

\section{The ring of meromorphic Siegel modular forms of level 5}\label{sec:meromorphic}
We consider the ring $M_*^!(\Gamma_0^{(2)}(5))$ of meromorphic Siegel modular forms of level $\Gamma_0^{(2)}(5)$ whose poles may lie only on the orbit $\mathcal{H}$ of the Humbert surface $$\{Z = \begin{psmallmatrix} \tau & z \\ z & w \end{psmallmatrix} \in \mathbb{H}_2: \; \mathrm{det}(Z) = 1 - 5z\},$$ which is a locus of principally polarized RM abelian surfaces with $\Gamma_0^{(2)}(5)$ level structure. In view of our discussion earlier, $\mathcal{H}$ splits as the union of two irreducible $\Gamma_{0, 1}^{(2)}(5)$-orbits of Humbert surfaces: $$\mathcal{H} = H(1/20, \pm \gamma_1) + H(1/20, \pm \gamma_2),$$ each invariant under the discriminant kernel of $L = U\oplus U(5)\oplus A_1$, where we have fixed any coset $\gamma_1 \in L'/L$ of norm $1/20 + \mathbb{Z}$ and define $\gamma_2 = \varepsilon_2(\gamma_1)$. The Humbert surface $H_{1/5}$ of discriminant $1/5$, the orbit of $\{\begin{psmallmatrix} \tau & z \\ z & w \end{psmallmatrix} \in \mathbb{H}_2: \; \tau = 2z\}$ under $\Gamma_0^{(2)}(5)$, also splits into two $\Gamma_{0, 1}^{(2)}(5)$-invariant divisors: $$H_{1/5} = H(1/5, \pm \delta_1) + H(1/5, \pm \delta_2),$$ where $\delta_n = 2\gamma_n \in L'/L$.

For a finite-index subgroup $\Gamma \le \Gamma_0^{(2)}(5)$ or $\Gamma \le \mathrm{O}(L)$, we define $M_*^!(\Gamma,\chi)$ to be the ring of meromorphic forms, holomorphic away from $\mathcal{H}$, which are modular under $\Gamma$ with character $\chi$.

We first prove a form of Koecher's principle for meromorphic modular forms with poles supported on $\mathcal{H}$.  

\begin{lem}\label{lem:Koecher}
Let $f\in M_k^!(\Gamma_{0,1}^{(2)}(5),\chi)$. If $k$ is negative, then $f$ is identically zero. If $k=0$, then $f$ is constant.
\end{lem}
\begin{proof}
We prove the lemma in the context of $\mathrm{O}(3, 2)$. Let $v$ and $u\neq \pm v$ be primitive vectors of norm $1/20$ in $L'$, such that $v^\perp, u^\perp \in \mathcal{H}$. Suppose that $f$ is not identically zero and has poles of multiplicity $c_v$ along $v^\perp$.  We denote the intersection of $v^\perp$ and the symmetric domain $\mathbb{H}_2$ (resp. the lattice $L$) by $v^\perp \cap \mathbb{H}_2$ (resp. $L_v$). Then $L_v$ is a lattice of signature $(2,2)$ and discriminant $5$, equivalent to the lattice $U + \mathbb{Z}[(1 + \sqrt{5}) / 2]$ where the quadratic form is the field norm. It is easy to see that the space $L_v\otimes \mathbb{Q}$ contains no isotropic planes, so the Koecher principle holds for modular forms on $\widetilde{\mathrm{O}}(L_v)$. We find that the projection of $u$ in $L_v$ has non-positive norm, which implies that the intersection of $u^\perp$ and $v^\perp \cap \mathbb{H}_2\cong \mathbb{H}\times \mathbb{H}$ is empty. Thus the quasi-pullback of $f$ to $v^\perp \cap \mathbb{H}_2$, i.e. the leading term in the power series expansion about that hyperplane, is a nonzero holomorphic modular form of weight $k-c_v$. By Koecher's principle we conclude $k-c_v\geq 0$ and therefore $k \ge 0$, and when $k=0$, we must have $c_v=0$ and thus $f$ is holomorphic and must be constant (by Koecher's principle on $\widetilde{\mathrm{O}}(L)$).
\end{proof}

We now construct some basic modular forms using Borcherds additive lifts (singular Gritsenko lifts) and Borcherds products. 

\begin{lem}\label{wt1} There are singular Gritsenko lifts $f_1, f_2$ of weight one on $\widetilde{\mathrm{O}}(L)$ whose divisors are exactly $$\mathrm{div}(f_1) = -H(1/20, \pm \gamma_1) + 4H(1/20, \pm \gamma_2) + H(1/5, \pm \delta_1)$$ and $$\mathrm{div}(f_2) = 4H(1/20, \pm \gamma_1) - H(1/20, \pm \gamma_2) + H(1/5, \pm \delta_2).$$
\end{lem}
\begin{proof} Using the algorithm of \cite{W} we find a nearly-holomorphic modular form of weight $1/2$ for the Weil representation associated to $L$ whose Fourier expansion takes the form $$2 q^{-1/20}(\mathfrak{e}_{\gamma_1} - \mathfrak{e}_{-\gamma_1}) + O(q^{1/20}),$$ which is mapped under the Gritsenko lift to a meromorphic form $f_1$ with simple poles only on $H(1/20, \pm \gamma_1)$ and $H(1/20, \pm \gamma_4)$. Applying the automorphism $\varepsilon_2$ on $L'/L$ to the input into $f_1$ yields the input into $f_2$.

On the other hand, we found a nearly-holomorphic modular form of weight $-1/2$ whose principal part at $\infty$ is $$2 \mathfrak{e}_0 - 2q^{-1/20} (\mathfrak{e}_{\gamma_1} + \mathfrak{e}_{-\gamma_1}) + 4q^{-1/20} (\mathfrak{e}_{\gamma_2} + \mathfrak{e}_{-\gamma_2}) + q^{-1/5} (\mathfrak{e}_{\delta_1} + \mathfrak{e}_{-\delta_1}),$$ which is mapped under the Borcherds lift to a meromorphic modular form $F_1$ (possibly with character) of weight one and the claimed divisor. By taking theta contractions of the input form one finds that $f_1$ vanishes on $H(1/5, \pm \delta_1)$. Then the quotient $f_1 / F_1$ lies in $M_0^!(\widetilde{\mathrm{O}}(L), \chi)$ so it is constant by Lemma \ref{lem:Koecher}.
\end{proof}

\begin{rem} The Fourier expansions of $f_1$ and $f_2$ begin 
\begin{align*}
f_1 \left( \begin{psmallmatrix} \tau & z \\ z & w \end{psmallmatrix}\right) &= 1 + 3q + 3s + 4q^2 + (2r^{-1} + 6 + 2r) qs + 4s^2 + O(q,s)^3;\\    
f_2 \left( \begin{psmallmatrix} \tau & z \\ z & w \end{psmallmatrix} \right) &= q - s - 2q^2 + 2s^2 + 4q^3 + (4r^2 + 2 + 4r) qs (q - s) - 4s^3 + O(q, s)^4;
\end{align*}
where as usual $q = e^{2\pi i \tau}$, $r = e^{2\pi i z}$, $s = e^{2\pi i w}$. For more coefficients see Figure \ref{tab:coeffs} below. Setting $s = 0$ one obtains the (holomorphic) modular forms $$\Phi(f_1) = 1 + 3q + 4q^2 \pm ..., \; \Phi(f_2) = q - 2q^2 + 4q^3 \pm ...$$ of weight one and level $\Gamma_1(5)$ which freely generate the ring $M_*(\Gamma_1(5))$.
\end{rem}

There are nine Heegner divisors of discriminant $1/4$. One is the mirror of the reflective vector $r = 1/2 \in A_1'$, represented by the diagonal in $\mathbb{H}_2$, and the other eight are of the form $H(1/4, r + \gamma)$ where $\gamma$ are the isotropic cosets of $U(5)'/U(5)$. It will be convenient to fix concrete representatives. We take the Gram matrix $\mathbf{S} = \begin{psmallmatrix} 0 & 0 & 5 \\ 0 & 2 & 0 \\ 5 & 0 & 0 \end{psmallmatrix}$ for $U(5) \oplus A_1$, such that $L'/L \cong \mathbf{S}^{-1} \mathbb{Z}^3 / \mathbb{Z}^3$ and fix the cosets \begin{align*} &\gamma_1 = (1/5, 1/2, 4/5) + L&  &\gamma_2 = (2/5, 1/2, 2/5) + L& \\ &\gamma_3 = (3/5, 1/2, 3/5) + L&  &\gamma_4 = (4/5, 1/2, 1/5) + L&\end{align*} of norm $1/20 + \mathbb{Z}$. The norm $1/4$ cosets other than $r$ are labelled $$\alpha_n = (n/5, 1/2, 0) + L, \; \beta_n = (0, 1/2, n/5) + L, \; n \in \{1, 2, 3, 4\}.$$

\begin{lem} There are singular Gritsenko lifts $g_1, g_2, h_1, h_2$ of weight two on $\widetilde{\mathrm{O}}(L)$ whose divisors are exactly \begin{align*} \mathrm{div}\, g_1 &= 3 H(1/20, \pm \gamma_1) - 2 H(1/20, \pm \gamma_2) + H(1/4, \pm \alpha_2); \\ \mathrm{div}\, g_2 &= -2H(1/20, \pm \gamma_1) + 3H(1/20, \pm \gamma_2) + H(1/4, \pm \alpha_1); \\ \mathrm{div}\, h_1 &= 3 H(1/20, \pm \gamma_1) - 2 H(1/20, \pm \gamma_2)+ H(1/4, \pm \beta_2); \\ \mathrm{div}\, h_2 &= -2H(1/20, \pm \gamma_1) + 3H(1/20, \pm \gamma_2) + H(1/4, \pm \beta_1). \end{align*}
\end{lem}
\begin{proof} The proof is essentially the same argument as Lemma \ref{wt1}. Using the pullback trick, one constructs weight two Gritsenko lifts which vanish on the claimed discriminant $1/4$ Heegner divisors. Then one constructs Borcherds products of weight two with the claimed divisors. The respective quotients lie in $M_0^!(\widetilde{\mathrm{O}}(L), \chi)$ and are therefore constant by Lemma \ref{lem:Koecher}. To determine the precise (nearly-holomorphic) vector-valued modular forms which lift to $g_1, g_2, h_1, h_2$, one only needs to compute the four-dimensional space of nearly-holomorphic forms of weight $3/2$ for $\rho_L$ with a pole of order at most $1/20$ at $\infty$, and identify the unique (up to scalar) forms whose pullback to $\alpha_n^{\perp}$ or $\beta_n^{\perp}$ is respectively zero. The input forms $G_1, G_2, H_1, H_2$ can be chosen such that their Fourier expansions begin as follows: \begin{align*} G_1: &\quad q^{-1/20} (\mathfrak{e}_{\gamma_2} + \mathfrak{e}_{\gamma_3}) +  (\mathfrak{e}_{(0, 0, 1/5)} + \mathfrak{e}_{(0, 0, 4/5)} - \mathfrak{e}_{(0, 0, 2/5)} + \mathfrak{e}_{(0, 0, 3/5)} ) + O(q^{1/20}) \\ G_2: &\quad q^{-1/20} (\mathfrak{e}_{\gamma_1} + \mathfrak{e}_{\gamma_4}) + (\mathfrak{e}_{(0, 0, 2/5)} + \mathfrak{e}_{(0, 0, 3/5)} - \mathfrak{e}_{(0, 0, 1/5)} + \mathfrak{e}_{(0, 0, 4/5)} ) + O(q^{1/20}) \\ H_1: &\quad q^{-1/20} (\mathfrak{e}_{\gamma_2} + \mathfrak{e}_{\gamma_3}) + (\mathfrak{e}_{(1/5, 0, 0)} + \mathfrak{e}_{(4/5, 0, 0)} - \mathfrak{e}_{(2/5, 0, 0)} + \mathfrak{e}_{(3/5, 0, 0)} ) + O(q^{1/20}) \\ H_2: &\quad q^{-1/20} (\mathfrak{e}_{\gamma_1} + \mathfrak{e}_{\gamma_4}) + (\mathfrak{e}_{(2/5, 0, 0)} + \mathfrak{e}_{(3/5, 0, 0)} - \mathfrak{e}_{(1/5, 0, 0)} + \mathfrak{e}_{(4/5, 0, 0)} )  + O(q^{1/20}). \end{align*} These expansions determine $G_1,G_2,H_1,H_2$ uniquely because there are no vector-valued cusp forms of weight $3/2$ for $\rho_L$.
\end{proof}

\begin{rem} The Fourier expansions of $g_1, g_2, h_1, h_2$ begin as follows:
\begin{align*}
g_1\left( \begin{psmallmatrix} \tau & z \\ z & w \end{psmallmatrix}\right) &= q + q^2 - 5qs + 2q^3 + (-3r^{-1} + 1 - 3r) q^2 s + (-r^{-1} + 7 - r) qs^2 + O(q, s)^4;\\
g_2\left( \begin{psmallmatrix} \tau & z \\ z & w \end{psmallmatrix}\right) &= -q - q^2 + (r^{-1} + 3 + r)qs - 2q^3 + (-r^{-1} + 7 - r) q^2 s + (-3r^{-1} + 1 - 3r) qs^2 + O(q, s)^4;\\
h_1 \left( \begin{psmallmatrix} \tau & z \\ z & w \end{psmallmatrix}\right) &= s - 5qs + s^2 + (-r^{-1} +7 - r) q^2 s + (-3r^{-1} + 1 - 3r) qs^2 + 2s^3 + O(q, s)^4;\\
h_2 \left( \begin{psmallmatrix} \tau & z \\ z & w \end{psmallmatrix}\right) &= -s + (r^{-1} + 3 + r) qs - s^3 + (-3r^{-1} + 1 - 3r) q^2 s + (-r^{-1} + 7 - r) qs^2 - 2s^3 + O(q, s)^4.
\end{align*}
\end{rem}

We can now determine the structure of $M_*^!(\Gamma_{0, 1}^{(2)}(5))$. Recall that $\Gamma_{0, 1}^{(2)}(5) /\{\pm I\} \cong \widetilde{\mathrm{SO}}(L)$. The decomposition
$$
M_k^!(\widetilde{\mathrm{SO}}(L))= M_k^!(\widetilde{\mathrm{O}}(L)) \oplus M_k^!(\widetilde{\mathrm{O}}(L), \det), 
$$
suggests that we first consider the ring of modular forms for the discriminant kernel $\widetilde{\mathrm{O}}(L)$. We will show that $M_*^!(\widetilde{\mathrm{O}}(L))$ is freely generated using a generalization of the modular Jacobian approach of \cite[Theorem 5.1]{Wan21a}. We briefly introduce the main objects of this approach. For any four $\psi_i\in M_{k_i}^!(\widetilde{\mathrm{O}}(L))$ with $1\leq i\leq 4$, their Jacobian (see \cite[Theorem 2.5]{Wan21a} and \cite[Proposition 2.1]{AI})
\begin{equation*}
J(\psi_1,\psi_2,\psi_3,\psi_4)=\left\lvert \begin{array}{cccc}
k_1\psi_1 & k_2\psi_2 & k_3\psi_3 & k_4\psi_4 \\ 
[1mm]
\frac{\partial \psi_1}{\partial \tau} & \frac{\partial \psi_2}{\partial \tau} & \frac{\partial \psi_3}{\partial \tau} & \frac{\partial \psi_4}{\partial \tau} \\[1mm]
\frac{\partial \psi_1}{\partial z} & \frac{\partial \psi_2}{\partial z} & \frac{\partial \psi_3}{\partial z} & \frac{\partial \psi_4}{\partial z} \\
[1mm]
\frac{\partial \psi_1}{\partial w} & \frac{\partial \psi_2}{\partial w} & \frac{\partial \psi_3}{\partial w} & \frac{\partial \psi_4}{\partial w} 
\end{array}   \right\rvert
\end{equation*}
lies in  $M_{k_1+k_2+k_3+k_4+3}^!(\widetilde{\mathrm{O}}(L),\det)$. The Jacobian $J(\psi_1,\psi_2,\psi_3,\psi_4)$ is not identically zero if and only if the four forms $\psi_i$ are algebraically independent over $\mathbb{C}$. 

The discriminant kernel $\widetilde{\mathrm{O}}(L)$ contains reflections associated to vectors of norm 1 in $L$ (the so-called $2$-reflections)
$$
\sigma_v: x \mapsto x - (x,v)v.
$$
The hyperplane $v^\perp$ is called the mirror of the reflection $\sigma_v$. Since $\det(\sigma_v)=-1$, the chain rule implies that the above Jacobian vanishes on all mirrors of $2$-reflections. Conversely, the main theorem of \cite{Wan21a}, and its generalization to meromorphic modular forms with constrained poles, implies that

\begin{thm}\label{th:singular1}
The ring $M_*^!(\widetilde{\mathrm{O}}(L))$ is a free algebra: $$M_*^!(\widetilde{\mathrm{O}}(L)) = \mathbb{C}[f_1, f_2, g_1, g_2].$$ 
Define $J:=J(f_1,f_2,g_1,g_2)$. Then
$$
M_*^!(\Gamma_{0, 1}^{(2)}(5)) = \mathbb{C}[f_1, f_2, g_1, g_2, J].
$$
\end{thm}
\begin{proof} The Jacobian $J$ of $f_1, f_2, g_1, g_2$ has weight $9$ and vanishes on the mirrors of $2$-reflections, which form a union of Heegner divisors of discriminants $1/4$ and $1$ denoted $\Delta$. Using the Fourier expansions of the forms it is easy to check that $J$ is not identically zero. Using the algorithm of \cite{W} we find a Borcherds product $J_0$ with divisor $$\mathrm{div}\, J_0 = \Delta + 6 H(1/20, \pm \gamma_1) + 6 H(1/20, \pm \gamma_2).$$ The quotient $J / J_0$ lies in $M_0^!(\Gamma_{0, 1}^{(2)}(5),\chi)$ and is therefore a constant denoted $c$ by Lemma \ref{lem:Koecher}. We will now prove the claim by an argument which appeared essentially in \cite{Wan21a}. Suppose that $M_*^!(\widetilde{\mathrm{O}}(L))$ was not generated by $h_1 := f_1$, $h_2 := f_2$, $h_3 := g_1$ and $h_4 := g_2$, and let $h_5 \in M_{k_5}^!(\widetilde{\mathrm{O}}(L))$ be a modular form of minimal weight which is not contained in $\mathbb{C}[f_1,f_2,g_1,g_2]$. Set $k_1 = k_2 = 1$ and $k_3 = k_4 = 2$, such that $k_i$ is the weight of $h_i$. For $1\leq j \leq 5$ we define $J_j$ as the Jacobian of the four modular forms $h_i$ omitting $h_j$, such that $c J_0=J=J_{5}$.  It is clear that $g_j:=J_j/J$ is a modular form on $\widetilde{\mathrm{O}}(L)$ with poles supported on $\mathcal{H}$. We compute the determinant and find the identity
$$
0 = \mathrm{det} \begin{pmatrix} k_1 h_1 & ... & k_4 h_4 & k_5 h_5 \\ k_1 h_1 & ... & k_4 h_4 & k_5 h_5 \\ \nabla h_1 & ... & \nabla h_4 & \nabla h_5 \end{pmatrix} = \sum_{i=1}^5 (-1)^{i+1} k_i h_i J_i.
$$
Since $J_i=Jg_i$ and $g_5=1$, we have 
$$
\sum_{i=1}^{5} (-1)^{i+1} k_t h_t g_t = 0, \quad  \text{i.e} \quad k_5 h_5 = \sum_{i=1}^4 (-1)^i h_i g_i.
$$
Since $h_5$ was chosen to have minimal weight, $g_i \in \mathbb{C}[h_1,h_2,h_3,h_4]$ for all $i$, and thus $h_5\in \mathbb{C}[h_1,h_2,h_3,h_4]$, which is a contradiction.

Now any $h\in M_k^!(\widetilde{\mathrm{O}}(L),\det)$ vanishes on all mirrors of $2$-reflections, which implies that $h/J \in M_{k-9}^!(\widetilde{\mathrm{O}}(L))$. Therefore $$M_*^!(\Gamma_{0,1}^{(2)}(5)) = M_*^!(\widetilde{\mathrm{SO}}(L)) = M_*^!(\widetilde{\mathrm{O}}(L)) \oplus M_*^!(\widetilde{\mathrm{O}}(L), \det)$$ is generated by $f_1$, $f_2$, $g_1$, $g_2$ and $J$.
\end{proof}

\begin{rem} 
The weight two singular Gritsenko lifts satisfy the relations $$g_1 - h_1 = h_2 - g_2 = f_1 f_2.$$ 
The product $f_1 f_2$ is holomorphic and in fact itself a Gritsenko lift; but it has a quadratic character under $\Gamma_0^{(2)}(5)$. There is a unique normalized Siegel modular form $e_2$ of weight two for $\Gamma_0^{(2)}(5)$, which can be constructed as the Gritsenko lift of the unique vector-valued modular form of weight $3/2$ for $\rho_L$ invariant under all automorphisms of the discriminant form. (The uniqueness follows from Corollary \ref{cor:level5gen} blow.) In terms of the generators of $M_*^!(\Gamma_{0, 1}^{(2)}(5))$, a computation shows $$e_2 = f_1^2 + f_2^2 - 4(g_1 + g_2).$$
\end{rem}

\begin{cor}\label{cor:level5gen} The ring $M_*^!(\Gamma_0^{(2)}(5))$ is minimally generated in weights $2, 2, 4, 4, 4, 4, 4, 11, 11, 11$ by the ten forms
\begin{align*}
&f_1^2 + f_2^2,& &e_2,& &f_1^2 g_1 + f_2^2 g_2,& &f_1 f_2 (g_1 - g_2),& &f_1 f_2 (f_1 - f_2)(f_1 + f_2),&\\
&f_1^2 f_2^2,& &g_1 g_2,& &J f_1 f_2,& &J (f_1^2 - f_2^2),& &J (g_1 - g_2).&    
\end{align*}
\end{cor}
\begin{proof} The group $\Gamma_0^{(2)}(5)$ is generated by the special discriminant kernel of $L$ and by the order four automorphism $\varepsilon_2$ which acts on the generators of $M_*^!(\Gamma_{0, 1}^{(2)}(5))$ by $$\varepsilon_2 : f_1 \mapsto f_2, \; f_2 \mapsto -f_1, \; g_1 \mapsto g_2, \; g_2 \mapsto g_1, \; J \mapsto -J$$ as one can see on the input functions into the Gritsenko lifts. We conclude the action of $\varepsilon_2$ on $J$ (as a Jacobian) from the actions of $\varepsilon_2$ on $f_1$, $f_2$, $g_1$ and $g_2$. The expressions in $f_1, f_2, g_1, g_2, J$ in the claim generate the ring of invariants under this action.  
\end{proof}

\begin{rem} The same argument shows that the kernel of $\varepsilon_2^2 = \varepsilon_4$ is generated by $J$ and by the weight two forms $$f_1^2, f_2^2, f_1 f_2, g_1, g_2.$$ This corresponds to the quadratic Nebentypus $\chi\left( \begin{psmallmatrix} a & b \\ c & d \end{psmallmatrix} \right) = \left( \frac{5}{\mathrm{det}\, d}\right)$ on $\Gamma_0^{(2)}(5)$. Note that $f_1 f_2$ is the Siegel Eisenstein series of weight two for the character $\chi$, and that the Jacobian $J$ is the unique cusp form of weight nine for $\chi$ up to scalars.
\end{rem}

\begin{rem}\label{rem:liftslevel5} There is a seven-dimensional space of modular forms of weight $7/2$ for $\rho_L$, and a four-dimensional subspace on which $\varepsilon_2$ acts trivially, so the weight four Maass space for $\Gamma_0^{(2)}(5)$ is four-dimensional. Using the structure theorem above we can identify it by comparing only a few Fourier coefficients: $$\mathrm{Maass}_4 = \mathrm{Span}(\phi_1, \phi_2, \phi_3, \phi_4)$$ where \begin{align*} \phi_1 &= e_2^2 + f_1^2 f_2^2; \\ \phi_2 &= f_1^2 g_1 + f_2^2 g_2 - 2g_1 g_2; \\ \phi_3 &= f_1 f_2 (f_1^2 - 2 f_1 f_2 - f_2^2 + 2 g_1 - 2g_2); \\ \phi_4 &= 2 g_1 g_2 + f_1 f_2 (g_1 - g_2). \end{align*} The form $\phi_4$ is a cusp form and indeed spans $S_4(\Gamma_0^{(2)}(5))$, which was shown to be one-dimensional by Poor and Yuen \cite{PY}.
\end{rem}

\section{The ring of holomorphic Siegel modular forms of level 5}\label{sec:holomorphic}
In this section we investigate the ring $M_*(\Gamma_0^{(2)}(5))$ of holomorphic Siegel modular forms for $\Gamma_0^{(2)}(5)$. We will need the Hilbert--Poincar\'e series for this ring, which can be derived from dimension formulas available in the literature.
\begin{thm}\label{thm:dim} The Hilbert--Poincar\'e series of dimensions of modular forms for $\Gamma_0^{(2)}(5)$ is $$\sum_{k = 0}^{\infty} \mathrm{dim}\, M_k(\Gamma_0^{(2)}(5)) t^k = \frac{(1 - t)^2 (1 + t^7) P(t)}{(1 - t^2)^2 (1 - t^3)(1 - t^4)^2 (1 - t^5)}$$ where $P(t)$ is the irreducible palindromic polynomial $$P(t) = 1 + 2 t + 2 t^{2} +  t^{3} + 3 t^{4} + 5 t^{5} + 8 t^{6} + 8 t^{7} + 8 t^{8} + 5 t^{9} + 3 t^{10} +  t^{11} + 2 t^{12} + 2 t^{13} +  t^{14}.$$
\end{thm}
The first values of $\dim M_k(\Gamma_0^{(2)}(5))$ are given in Table \ref{tab} below.

\begin{table}[ht]
\caption{$\dim M_k(\Gamma_0^{(2)}(5))$}\label{tab}
\renewcommand\arraystretch{1.5}
\noindent\[
\begin{array}{|c|c|c|c|c|c|c|c|c|c|c|c|c|c|c|c|c|c|c|c|}
\hline 
k & 1 & 2 & 3 & 4 & 5 & 6 & 7 & 8 & 9 & 10 & 11 & 12 & 13 & 14 &15 & 16 &17 &18 & 19 \\ 
\hline 
\text{dim} & 0 & 1 & 0 & 6 & 0 & 10 & 0 & 22 & 0 & 34 & 3 & 57 & 6 & 79 & 16 & 117 & 25 & 153 & 45  \\ 
\hline 
\end{array} 
\]
\end{table}

\begin{proof} The dimensions of the spaces of cusp forms of weight $k \ge 5$ have been computed in closed form by Hashimoto by means of the Selberg trace formula and in lower weights by Poor and Yuen \cite{PY}: we have $\mathrm{dim}\, S_4(\Gamma_0^{(2)}(5)) = 1$ and $\mathrm{dim}\, S_k(\Gamma_0^{(2)}(5)) = 0$ for $k \le 3$. All odd-weight modular forms are cusp forms, and by a more general theorem of B\"ocherer--Ibukiyama \cite{BI}, for even $k > 2$, $$\mathrm{dim}\, M_k(\Gamma_0^{(2)}(5)) = \mathrm{dim}\, S_k(\Gamma_0^{(2)}(5)) + 2 \cdot \mathrm{dim}\, S_k(\Gamma_0(5)) + 3.$$
\end{proof}
We can now determine the generators of $M_*(\Gamma_0^{(2)}(5))$ using Corollary \ref{cor:level5gen} together with the above generating series.

\begin{thm}\label{thm:lvl5hol} The ring of Siegel modular forms of level $\Gamma_0^{(2)}(5)$ is minimally generated by the weight two form $$e_2 = f_1^2 + f_2^2 - 4 g_1 - 4g_2,$$ five weight four forms $$f_1^2 g_1 + f_2^2 g_2, \quad f_1 f_2 (g_1 - g_2), \quad f_1 f_2 (f_1^2 - f_2^2), \quad f_1^2 f_2^2, \quad g_1 g_2,$$ four weight six forms $$f_1^2 f_2^2 (g_1 + g_2), \quad f_1^3 f_2 g_1 - f_2^3 f_1 g_2, \quad f_1^2 g_1^2 + f_2^2 g_2^2, \quad g_1 g_2 (f_1^2 + f_2^2),$$ the weight ten form $$f_1^2 f_2^2 (f_1^2 + f_2^2)^3,$$ three weight eleven forms $$f_1 f_2 J, \quad (f_1^2 - f_2^2) J, \quad  (g_1 - g_2) J,$$ three weight thirteen forms $$(f_1^2 + f_2^2) f_1 f_2 J, \quad (f_1^4 - f_2^4) J, \quad (f_1^2 + f_2^2)(g_1 - g_2) J,$$ and the weight fifteen form $$(f_1^2 - f_2^2)^3 J.$$
\end{thm}
\begin{proof} From the divisors of $f_1, f_2, g_1, g_2$ and $J$, it is easy to see that all of the forms above (except for $e_2$, which was discussed in the previous section) are holomorphic and $\varepsilon_2$-invariant. The Hilbert series of this ring was computed in Macaulay2 \cite{M2} and coincides exactly with the series predicted by Theorem \ref{thm:dim}, so we can conclude that these forms are sufficient to generate all holomorphic Siegel modular forms.
\end{proof}

\bigskip

\noindent
\textbf{Acknowledgements} 
H. Wang was supported by the Institute for Basic Science (IBS-R003-D1).

\section{Appendix}

On the following pages, we list more Fourier coefficients of the basic meromorphic forms $f_1, f_2, g_1, g_2$, as well as the unique expression for $J^2$ as a polynomial in these forms.

Note that the polynomial representing $J^2$ must split into two irreducible factors, corresponding to the two classes of reflections whose mirrors lie in the divisor of $J$. One of these factors is $g_1 + g_2$.

\begin{figure}
    \centering
        
    \begin{tabular}{m|m|m||m|m|m|m}
        a & b & c & f_1 & f_2 & g_1 & g_2  \\ \hline
        0 & 0 & 0 & 1 & 0 & 0 & 0 \\ 1 & 0 & 0 & 3 & 1 & 1 & -1 \\0 & 0 & 1 & 3 & -1 & 0 & 0 \\2 & 0 & 0 & 4 & -2 & 1 & -1 \\1 & -1 & 1 & 2 & 0 & 0 & 1 \\1 & 0 & 1 & 6 & 0 & -5 & 3 \\1 & 1 & 1 & 2 & 0 & 0 & 1 \\0 & 0 & 2 & 4 & 2 & 0 & 0 \\3 & 0 & 0 & 2 & 4 & 2 & -2 \\2 & -1 & 1 & 4 & -4 & -3 & -1 \\2 & 0 & 1 & 2 & -2 & 1 & 7 \\2 & 1 & 1 & 4 & -4 & -3 & -1 \\1 & -1 & 2 & 4 & 4 & -1 & -3 \\1 & 0 & 2 & 2 & 2 & 7 & 1 \\1 & 1 & 2 & 4 & 4 & -1 & -3 \\0 & 0 & 3 & 2 & -4 & 0 & 0 \\4 & 0 & 0 & 1 & -3 & 3 & -3 \\3 & -1 & 1 & 6 & 6 & -2 & 4 \\3 & 0 & 1 & -2 & -2 & -6 & 2 \\3 & 1 & 1 & 6 & 6 & -2 & 4 \\2 & -2 & 2 & 2 & 0 & 0 & 2 \\2 & -1 & 2 & 8 & 0 & 10 & -2 \\2 & 0 & 2 & 0 & 0 & -15 & -5 \\2 & 1 & 2 & 8 & 0 & 10 & -2 \\2 & 2 & 2 & 2 & 0 & 0 & 2 \\1 & -1 & 3 & 6 & -6 & 4 & -2 \\1 & 0 & 3 & -2 & 2 & 2 & -6 \\1 & 1 & 3 & 6 & -6 & 4 & -2 \\0 & 0 & 4 & 1 & 3 & 0 & 0 \\5 & 0 & 0 & 3 & 1 & 5 & -5 \\4 & -1 & 1 & 0 & -8 & -2 & 10 \\4 & 0 & 1 & 0 & 6 & -11 & -5 \\4 & 1 & 1 & 0 & -8 & -2 & 10 \\3 & -2 & 2 & 0 & 6 & 3 & -5 \\3 & -1 & 2 & 0 & 0 & -10 & 10 \\3 & 0 & 2 & 0 & 8 & 24 & -20 \\3 & 1 & 2 & 0 & 0 & -10 & 10 \\3 & 2 & 2 & 0 & 6 & 3 & -5 \\2 & -2 & 3 & 0 & -6 & 3 & -5 \\2 & -1 & 3 & 0 & 0 & -10 & 10 \\2 & 0 & 3 & 0 & -8 & 24 & -20 \\2 & 1 & 3 & 0 & 0 & -10 & 10 \\2 & 2 & 3 & 0 & -6 & 3 & -5 \\1 & -1 & 4 & 0 & 8 & -2 & 10 \\1 & 0 & 4 & 0 & -6 & -11 & -5 \\1 & 1 & 4 & 0 & 8 & -2 & 10 \\0 & 0 & 5 & 3 & -1 & 0 & 0 \\ 6 & 0 & 0 & 6 & 2 & 2 & -2 \\5 & -2 & 1 & 3 & -1 & 0 & 0 \\5 & -1 & 1 & -2 & 4 & -15 & 10 \\5 & 0 & 1 & 8 & -6 & 5 & 5 \\5 & 1 & 1 & -2 & 4 & -15 & 10 \\5 & 2 & 1 & 3 & -1 & 0 & 0 \\4 & -2 & 2 & 6 & -6 & -4 & -8 \\4 & -1 & 2 & -8 & 8 & 12 & 4 \\4 & 0 & 2 & 14 & -14 & -1 & -7 
    \end{tabular} 
    \quad
    \begin{tabular}{m|m|m||m|m|m|m}
        a & b & c & f_1 & f_2 & g_1 & g_2 \\ \hline
        4 & 1 & 2 & -8 & 8 & 12 & 4 \\ 4 & 2 & 2 & 6 & -6 & -4 & -8 \\ 3 & -3 & 3 & 2 & 0 & 0 & 3 \\3 & -2 & 3 & 6 & 0 & -5 & -11 \\3 & -1 & 3 & -6 & 0 & 40 & 1 \\3 & 0 & 3 & 16 & 0 & -50 & -6 \\3 & 1 & 3 & -6 & 0 & 40 & 1 \\3 & 2 & 3 & 6 & 0 & -5 & -11 \\3 & 3 & 3 & 2 & 0 & 0 & 3 \\2 & -2 & 4 & 6 & 6 & -8 & -4 \\2 & -1 & 4 & -8 & -8 & 4 & 12 \\2 & 0 & 4 & 14 & 14 & -7 & -1 \\2 & 1 & 4 & -8 & -8 & 4 & 12 \\2 & 2 & 4 & 6 & 6 & -8 & -4 \\1 & -2 & 5 & 3 & 1 & 1 & -1 \\1 & -1 & 5 & -2 & -4 & -9 & 4 \\1 & 0 & 5 & 8 & 6 & 11 & -1 \\1 & 1 & 5 & -2 & -4 & -9 & 4 \\1 & 2 & 5 & 3 & 1 & 1 & -1 \\0 & 0 & 6 & 6 & -2 & 0 & 0 \\7 & 0 & 0 & 4 & -2 & 6 & -6 \\6 & -2 & 1 & 6 & 0 & -5 & 3 \\6 & -1 & 1 & 0 & 0 & 10 & -10 \\6 & 0 & 1 & 8 & 0 & -20 & 24 \\6 & 1 & 1 & 0 & 0 & 10 & -10 \\6 & 2 & 1 & 6 & 0 & -5 & 3 \\5 & -2 & 2 & 6 & 8 & 5 & 5 \\5 & -1 & 2 & -4 & -12 & 15 & -35 \\5 & 0 & 2 & 6 & 18 & -15 & 35 \\5 & 1 & 2 & -4 & -12 & 15 & -35 \\5 & 2 & 2 & 6 & 8 & 5 & 5 \\4 & -3 & 3 & -4 & -4 & -1 & -3 \\4 & -2 & 3 & 6 & 6 & 24 & 12 \\4 & -1 & 3 & -12 & -12 & -43 & -49 \\4 & 0 & 3 & 10 & 10 & 70 & 50 \\4 & 1 & 3 & -12 & -12 & -43 & -49 \\4 & 2 & 3 & 6 & 6 & 24 & 12 \\4 & 3 & 3 & -4 & -4 & -1 & -3 \\3 & -3 & 4 & -4 & 4 & -3 & -1 \\3 & -2 & 4 & 6 & -6 & 12 & 24 \\3 & -1 & 4 & -12 & 12 & -49 & -43 \\3 & 0 & 4 & 10 & -10 & 50 & 70 \\3 & 1 & 4 & -12 & 12 & -49 & -43 \\3 & 2 & 4 & 6 & -6 & 12 & 24 \\3 & 3 & 4 & -4 & 4 & -3 & -1 \\2 & -2 & 5 & 6 & -8 & -1 & 11 \\2 & -1 & 5 & -4 & 12 & 9 & -29 \\2 & 0 & 5 & 6 & -18 & -21 & 41 \\2 & 1 & 5 & -4 & 12 & 9 & -29 \\2 & 2 & 5 & 6 & -8 & -1 & 11 \\1 & -2 & 6 & 6 & 0 & -5 & 3 \\1 & -1 & 6 & 0 & 0 & 10 & -10 \\1 & 0 & 6 & 8 & 0 & -20 & 24 \\1 & 1 & 6 & 0 & 0 & 10 & -10 \\1 & 2 & 6 & 6 & 0 & -5 & 3 \\0 & 0 & 7 & 4 & 2 & 0 & 0
    \end{tabular} 
    \caption{Fourier coefficients of $\begin{psmallmatrix} a & b/2 \\ b/2 & c \end{psmallmatrix}$ in the basic forms $f_1, f_2, g_1, g_2$, $a + c \le 7$}
    \label{tab:coeffs}
\end{figure}
\clearpage
\begin{figure}
    \begin{align*} &J^2 = \\
& f_{1}^{8} f_{2}^{2} g_{1}^{4} + 22 f_{1}^{7} f_{2}^{3} g_{1}^{4} + 119 f_{1}^{6} f_{2}^{4} g_{1}^{4} - 22 f_{1}^{5} f_{2}^{5} g_{1}^{4} + f_{1}^{4} f_{2}^{6} g_{1}^{4} + 2 f_{1}^{8} f_{2}^{2} g_{1}^{3} g_{2} + 46 f_{1}^{7} f_{2}^{3} g_{1}^{3} g_{2} + 282 f_{1}^{6} f_{2}^{4} g_{1}^{3} g_{2} \\
& + 194 f_{1}^{5} f_{2}^{5} g_{1}^{3} g_{2} - 42 f_{1}^{4} f_{2}^{6} g_{1}^{3} g_{2} + 2 f_{1}^{3} f_{2}^{7} g_{1}^{3} g_{2} + f_{1}^{8} f_{2}^{2} g_{1}^{2} g_{2}^{2} + 22 f_{1}^{7} f_{2}^{3} g_{1}^{2} g_{2}^{2} + 120 f_{1}^{6} f_{2}^{4} g_{1}^{2} g_{2}^{2} + 120 f_{1}^{4} f_{2}^{6} g_{1}^{2} g_{2}^{2} \\
& - 22 f_{1}^{3} f_{2}^{7} g_{1}^{2} g_{2}^{2} + f_{1}^{2} f_{2}^{8} g_{1}^{2} g_{2}^{2} - 2 f_{1}^{7} f_{2}^{3} g_{1} g_{2}^{3} - 42 f_{1}^{6} f_{2}^{4} g_{1} g_{2}^{3} - 194 f_{1}^{5} f_{2}^{5} g_{1} g_{2}^{3} + 282 f_{1}^{4} f_{2}^{6} g_{1} g_{2}^{3} - 46 f_{1}^{3} f_{2}^{7} g_{1} g_{2}^{3} \\
& + 2 f_{1}^{2} f_{2}^{8} g_{1} g_{2}^{3} + f_{1}^{6} f_{2}^{4} g_{2}^{4} + 22 f_{1}^{5} f_{2}^{5} g_{2}^{4} + 119 f_{1}^{4} f_{2}^{6} g_{2}^{4} - 22 f_{1}^{3} f_{2}^{7} g_{2}^{4} + f_{1}^{2} f_{2}^{8} g_{2}^{4} - 2 f_{1}^{7} f_{2} g_{1}^{5} - 110 f_{1}^{6} f_{2}^{2} g_{1}^{5} \\
& - 166 f_{1}^{5} f_{2}^{3} g_{1}^{5} - 12 f_{1}^{4} f_{2}^{4} g_{1}^{5} - 4 f_{1}^{7} f_{2} g_{1}^{4} g_{2} - 242 f_{1}^{6} f_{2}^{2} g_{1}^{4} g_{2} - 450 f_{1}^{5} f_{2}^{3} g_{1}^{4} g_{2} - 438 f_{1}^{4} f_{2}^{4} g_{1}^{4} g_{2} - 24 f_{1}^{3} f_{2}^{5} g_{1}^{4} g_{2} \\
& - 2 f_{1}^{7} f_{2} g_{1}^{3} g_{2}^{2} - 142 f_{1}^{6} f_{2}^{2} g_{1}^{3} g_{2}^{2} + 62 f_{1}^{5} f_{2}^{3} g_{1}^{3} g_{2}^{2} + 110 f_{1}^{4} f_{2}^{4} g_{1}^{3} g_{2}^{2} - 370 f_{1}^{3} f_{2}^{5} g_{1}^{3} g_{2}^{2} - 10 f_{1}^{2} f_{2}^{6} g_{1}^{3} g_{2}^{2} - 10 f_{1}^{6} f_{2}^{2} g_{1}^{2} g_{2}^{3} \\
& + 370 f_{1}^{5} f_{2}^{3} g_{1}^{2} g_{2}^{3} + 110 f_{1}^{4} f_{2}^{4} g_{1}^{2} g_{2}^{3} - 62 f_{1}^{3} f_{2}^{5} g_{1}^{2} g_{2}^{3} - 142 f_{1}^{2} f_{2}^{6} g_{1}^{2} g_{2}^{3} + 2 f_{1} f_{2}^{7} g_{1}^{2} g_{2}^{3} + 24 f_{1}^{5} f_{2}^{3} g_{1} g_{2}^{4} - 438 f_{1}^{4} f_{2}^{4} g_{1} g_{2}^{4} \\
& + 450 f_{1}^{3} f_{2}^{5} g_{1} g_{2}^{4} - 242 f_{1}^{2} f_{2}^{6} g_{1} g_{2}^{4} + 4 f_{1} f_{2}^{7} g_{1} g_{2}^{4} - 12 f_{1}^{4} f_{2}^{4} g_{2}^{5} + 166 f_{1}^{3} f_{2}^{5} g_{2}^{5} - 110 f_{1}^{2} f_{2}^{6} g_{2}^{5} + 2 f_{1} f_{2}^{7} g_{2}^{5} + f_{1}^{6} g_{1}^{6} \\
& + 152 f_{1}^{5} f_{2} g_{1}^{6} + 48 f_{1}^{4} f_{2}^{2} g_{1}^{6} + 2 f_{1}^{6} g_{1}^{5} g_{2} + 340 f_{1}^{5} f_{2} g_{1}^{5} g_{2} + 300 f_{1}^{4} f_{2}^{2} g_{1}^{5} g_{2} + 96 f_{1}^{3} f_{2}^{3} g_{1}^{5} g_{2} + f_{1}^{6} g_{1}^{4} g_{2}^{2} + 212 f_{1}^{5} f_{2} g_{1}^{4} g_{2}^{2} \\
& - 312 f_{1}^{4} f_{2}^{2} g_{1}^{4} g_{2}^{2} - 394 f_{1}^{3} f_{2}^{3} g_{1}^{4} g_{2}^{2} + 24 f_{1}^{2} f_{2}^{4} g_{1}^{4} g_{2}^{2} + 24 f_{1}^{5} f_{2} g_{1}^{3} g_{2}^{3} - 540 f_{1}^{4} f_{2}^{2} g_{1}^{3} g_{2}^{3} - 540 f_{1}^{2} f_{2}^{4} g_{1}^{3} g_{2}^{3} - 24 f_{1} f_{2}^{5} g_{1}^{3} g_{2}^{3}\\
& + 24 f_{1}^{4} f_{2}^{2} g_{1}^{2} g_{2}^{4} + 394 f_{1}^{3} f_{2}^{3} g_{1}^{2} g_{2}^{4} - 312 f_{1}^{2} f_{2}^{4} g_{1}^{2} g_{2}^{4} - 212 f_{1} f_{2}^{5} g_{1}^{2} g_{2}^{4} + f_{2}^{6} g_{1}^{2} g_{2}^{4} - 96 f_{1}^{3} f_{2}^{3} g_{1} g_{2}^{5} + 300 f_{1}^{2} f_{2}^{4} g_{1} g_{2}^{5} \\
& - 340 f_{1} f_{2}^{5} g_{1} g_{2}^{5} + 2 f_{2}^{6} g_{1} g_{2}^{5} + 48 f_{1}^{2} f_{2}^{4} g_{2}^{6} - 152 f_{1} f_{2}^{5} g_{2}^{6} + f_{2}^{6} g_{2}^{6} - 64 f_{1}^{4} g_{1}^{7} - 144 f_{1}^{4} g_{1}^{6} g_{2} - 128 f_{1}^{3} f_{2} g_{1}^{6} g_{2} \\
& - 92 f_{1}^{4} g_{1}^{5} g_{2}^{2} + 264 f_{1}^{3} f_{2} g_{1}^{5} g_{2}^{2} + 32 f_{1}^{2} f_{2}^{2} g_{1}^{5} g_{2}^{2} - 12 f_{1}^{4} g_{1}^{4} g_{2}^{3} + 296 f_{1}^{3} f_{2} g_{1}^{4} g_{2}^{3} + 4 f_{1}^{2} f_{2}^{2} g_{1}^{4} g_{2}^{3} + 96 f_{1} f_{2}^{3} g_{1}^{4} g_{2}^{3} \\
& - 96 f_{1}^{3} f_{2} g_{1}^{3} g_{2}^{4} + 4 f_{1}^{2} f_{2}^{2} g_{1}^{3} g_{2}^{4} - 296 f_{1} f_{2}^{3} g_{1}^{3} g_{2}^{4} - 12 f_{2}^{4} g_{1}^{3} g_{2}^{4} + 32 f_{1}^{2} f_{2}^{2} g_{1}^{2} g_{2}^{5} - 264 f_{1} f_{2}^{3} g_{1}^{2} g_{2}^{5} - 92 f_{2}^{4} g_{1}^{2} g_{2}^{5} \\
& + 128 f_{1} f_{2}^{3} g_{1} g_{2}^{6} - 144 f_{2}^{4} g_{1} g_{2}^{6} - 64 f_{2}^{4} g_{2}^{7} - 128 f_{1}^{2} g_{1}^{6} g_{2}^{2} - 80 f_{1}^{2} g_{1}^{5} g_{2}^{3} - 128 f_{1} f_{2} g_{1}^{5} g_{2}^{3} + 48 f_{1}^{2} g_{1}^{4} g_{2}^{4} + 48 f_{2}^{2} g_{1}^{4} g_{2}^{4} \\
& + 128 f_{1} f_{2} g_{1}^{3} g_{2}^{5} - 80 f_{2}^{2} g_{1}^{3} g_{2}^{5} - 128 f_{2}^{2} g_{1}^{2} g_{2}^{6} - 64 g_{1}^{5} g_{2}^{4} - 64 g_{1}^{4} g_{2}^{5}. \end{align*}
    \caption{Expression for $J^2$ in terms of the basic forms $f_1, f_2, g_1, g_2$}
    \label{fig:my_label}
\end{figure}

\bibliographystyle{plainnat}
\bibliofont
\bibliography{refs}

\end{document}